\documentclass[11pt,a4paper]{amsart}
\usepackage[margin=3cm]{geometry}
\usepackage{color}               
\usepackage{faktor}
\usepackage{tikz}
\usepackage{tikz-cd}
\usepackage{spverbatim}
\usepackage{pstricks}
\usepackage{pdftricks}

\usepackage[pdftex,hyperfootnotes=false,colorlinks=true,linkcolor=blue,
citecolor=purple,filecolor=magenta,urlcolor=cyan,hypertexnames=false]{hyperref}

\usepackage{amsmath,amssymb,anyfontsize,enumerate,stmaryrd,mathtools, thmtools,tikz,txfonts,nccmath}
\usepackage[oldstyle]{libertine}%
\usepackage[T1]{fontenc}
\usepackage[utf8]{inputenc}
\usepackage{csquotes}
\makeatletter
\renewcommand*\libertine@figurestyle{LF}
\makeatother
\usepackage[libertine,libaltvw,liby]{newtxmath}
\makeatletter
\renewcommand*\libertine@figurestyle{OsF}
\makeatother
\usepackage{tikz-cd}
\usepackage[noabbrev]{cleveref}
\usepackage{autonum}
\crefname{lemma}{lemma}{lemmata}
\Crefname{lemma}{Lemma}{Lemmata}
\crefname{subsection}{subsection}{subsections}
\Crefname{subsection}{Subsection}{Subsections}

\newtheorem{theorem}{Theorem}[section]
\newtheorem{lemma}[theorem]{Lemma}

\newtheorem{corollary}[theorem]{Corollary}

\theoremstyle{definition}
\newtheorem{definition}[theorem]{Definition}
\newtheorem{remark}[theorem]{Remark}

\newtheorem{example}[theorem]{Example}

\newcommand{\R}{\mb{R}}

\newcommand{\Z}{\mb{Z}}

\newcommand{\cor}[1]{\bigg{\langle} \,  #1 \, \bigg{\rangle}}

\DeclareMathOperator\Aut{Aut}
\DeclareMathOperator\val{val}

\def\Z{\mathbb{Z}}

\def\R{\mathbb{R}}

\def\CP1{\mathbb{C}\mathrm{P}^1}

\DeclareMathOperator{\vir}{vir}

\newcommand\restr[2]{{% we make the whole thing an ordinary symbol
  \left.\kern-\nulldelimiterspace % automatically resize the bar with \right
  #1 % the function
  \vphantom{\big|} % pretend it's a little taller at normal size
  \right|_{#2} % this is the delimiter
  }}

\newcommand\restrst[2]{{% we make the whole thing an ordinary symbol
  \left.\kern-\nulldelimiterspace % automatically resize the bar with \right
  #1 % the function
  \vphantom{\big|} % pretend it's a little taller at normal size
  \right|^*_{#2} % this is the delimiter
  }}

\begingroup\expandafter\expandafter\expandafter\endgroup
\expandafter\ifx\csname pdfsuppresswarningpagegroup\endcsname\relax
\else
\fi
{}

\title{Wall-crossing and recursion formulae for tropical Jucys covers}
\author[M.~A.~Hahn]{Marvin Anas Hahn}
\address{M.~A.~Hahn: Institut für Mathematik, Goethe-Universität Frankfurt, Robert-Mayer-Str. 6-8, 60325 Frankfurt am Main}
\email{hahn@math.uni-frankfurt.de}
\author[D.~Lewa\'{n}ski]{Danilo Lewa\'{n}ski}
\address{D.~L.: Max Planck Institut f\"{u}r Mathematik, Vivatsgasse 7, 53111 Bonn, Germany.}
\email{ilgrillodani@mpim-bonn.mpg.de}


\begin{document}
\begin{abstract}
In recent work, the authors derived a tropical interpretation of monotone and strictly monotone double Hurwitz numbers. In this paper, we apply the technique of tropical flows to this interpretation in order to provide a new proof of the piecewise polynomiality of these enumerative invariants. Moreover, we derive new types of wall-crossing formulae.
\end{abstract}

\maketitle
\tableofcontents

\section{Introduction}
Hurwitz numbers \cite{Hurwitz} count branched genus $g$ coverings of the projective line with fixed ramification data. These objects connect several areas of mathematics, such as algebraic geometry, representation theory, mathematical physics and many more. In particular, they admit several equivalent definitions, among which is an interpretation due to Hurwitz in terms of factorisations in the symmetric group \cite{Hurwitz2}. From this interpretation many variants of Hurwitz numbers arise by imposing additional conditions on the factorisations. In this paper, we focus on two such variants, namely \textit{monotone} and \textit{strictly monotone Hurwitz numbers}. Monotone Hurwitz numbers were introduced in \cite{goulden2014monotone} in the context random matrix theory as the coefficients in the asymptotic expansion of the HCIZ integral, while strictly monotone Hurwitz numbers are equivalent to counting certain Grothendieck dessins d'enfants \cite{ALS}.\par 
In studying Hurwitz numbers, one often restricts onself to special allowed types of ramification. An important case is the one of \textit{single Hurwitz numbers}, where one allows arbitrary ramification over $\infty$ but only simple ramification (i.e. ramification profile $(2,1,\dots,1)$) over $b$ other points, where $b$ is determined by the Riemann-Hurwitz formula. These numbers admit a stunning connection to Gromov-Witten theory: the celebrated ELSV formula expresses single Hurwitz numbers in terms of intersection numbers on the moduli space of stable curves with marked points $\overline{\mathcal{M}}_{g,n}$ \cite{ekedahl2001hurwitz}. As a direct consequence single Hurwitz numbers are polynomial in the ramification profile over $\infty$ up to a combinatorial factor.\par 
From the study of single Hurwitz numbers, it is natural to consider arbitrary ramification over two points and simple ramification else. The numbers one obtains this way are called \textit{double Hurwitz numbers}. It is an open question whether double Hurwitz numbers satisfy an ELSV-type formula, i.e. an expression in terms of intersection numbers on some moduli space. One idea to approach this problem was introduced by Goulden, Jackson and Vakil in \cite{goulden2005towards}. Namely, one studies double Hurwitz numbers with a view towards polynomial behaviour. This may give an indication of the shape of an ELSV-type formula. In their work Goulden, Jackson and Vakil observe that double Hurwitz numbers are piecewise polynomial in the entries of the two arbitrary ramification profiles and determine the chambers of polynomiality. We note that this polynomiality is not up to a combinatorial factor. This leads them to a concrete conjecture on the shape of the ELSV-type formulae with the condition that all covers are fully ramified over $\infty$ which they prove for genus $0$ and genus $1$.\par 
This piecewise polynomial behaviour was further studied in work of Shadrin, Shapiro and Vainshtein, where it was observed that in genus $0$ the difference of the polynomials in two adjacent chambers may be expressed in terms of Hurwitz numbers with smaller input data \cite{shadrin2008chamber}. This was generalised to arbitrary genus by Cavalieri, Johnson and Markwig in \cite{CJMa} using tropical geometry and by Johnson in \cite{johnson2015double} in terms of the fermionic Fock space formalism.

\subsection{(Strictly) monotone double Hurwitz numbers}
In recent years it was shown in several instances that (strictly) monotone Hurwitz numbers share many features with their classical counterparts. For example, single monotone Hurwitz numbers satisfy an ELSV-type formula \cite{ALS}, the so-called Chekhov-Eynard-Orantin (CEO) topological recursion \cite{ALS}, and strictly monotone Hurwitz numbers satisfy CEO topological recursion in the so-called \textit{orbifold case} \cite{dumitrescu2013spectral,dunin2017combinatorics,norbury2013string, do2013quantum,kazarian2015virasoro}. Moreover, it was proved in \cite{DK,hahn2017monodromy} that (strictly) monotone double Hurwitz numbers are related to tropical geometry.  More precisely, there is an expression in terms of \textit{combinatorial covers} which are graphs related to tropical covers but decorated with extra combinatorial data. A common theme in studying (strictly) monotone double Hurwitz numbers is to consider some refinement of the enumeration and obtaining results for this refinement. An important example is the study of recursive behaviour of monotone Hurwitz numbers. A recursion for single monotone Hurwitz numbers was proved in \cite{dunin2019cut, goulden2013monotone}, while a recursion for monotone orbifold Hurwitz numbers and monotone double Hurwitz numbers remains an open question. However, it is possible to express monotone orbifold/double Hurwitz numbers as a sum of enumerations and deriving recursions for each summand. This approach was taken in \cite{DK} for the monotone orbifold Hurwitz numbers and in \cite{hahn2019triply} for the monotone double Hurwitz numbers, where each summand correspond to certain decorations on the combinatorial covers.\par 
In \cite{goulden2016toda}, it was proved that monotone double Hurwitz numbers behave piecewise polynomially with the same chambers of polynomials as the usual double Hurwitz numbers. This polynomial behaviour was further studied by the first author in \cite{hahn2017monodromy}, in terms of the aforementioned combinatorial covers. Using Ehrhart theory, algorithms were developed which compute the polynomials for monotone double Hurwitz numbers. We note that a priori these algorithms compute quasi-polynomials in a chamber structure much finer than necessary. In other words, the polynomial structure of monotone double Hurwitz numbers is not fully visible from this tropical viewpoint. However, it was possible to derive wall-crossing formulae in genus $0$.\par 
Motivated by the work in \cite{johnson2015double}, Kramer and the authors studied the piecewise polynomial behaviour of (strictly) monotone double Hurwitz numbers in the fermionic Fock space formalism in \cite{HKL}. In particular, it was proved that strictly monotone double Hurwitz numbers are piecewise polynomial which was an open question at the time. Moreover, a refinement of the generating series of (strictly) monotone double Hurwitz numbers was introduced, i.e. a larger generating series which specialises to the generating of (strictly) monotone double Hurwitz numbers. It was proved that this refinement admits wall-crossing formulae.

\subsection{Results}
In \cite{HL}, the authors derived a new interpretation of monotone and strictly monotone double Hurwitz numbers in terms of tropical covers which are weighted by Gromov-Witten invariants without any additional combinatorial decoration. In this paper, we use this new interpretation and apply the methods developed in \cite{CJMa} to study the wall-crossing behaviour of (strictly) monotone double Hurwitz numbers in arbitrary genus. In a sense, we take an opposite approach to \cite{HKL}. In \cite{HKL}, the generating series computing (strictly) monotone double Hurwitz numbers was enlarged and wall-crossing formulae were derived for this enlarged series. In this paper, we observe that using this new tropical interpretation, (strictly) monotone double Hurwitz numbers may naturally be written as a sum of smaller invariants, which we call $\lambda$\textit{-invariants}. These $\lambda-$invariants correspond to (ordered) partitions of the number of intermediate simple branch points and can be expressed as vacuum expectations of certain operators in the bosonic Fock space formalism and are thus not just obtained by combinatorial data. Moreover, it was proved in \cite[Theorem 5.10]{hahn2019triply} that the generating series of these invariants for elliptic base curves yield quasimodular forms.\par 
In \cref{thm:poly}, we prove that the $\lambda-$invariants are piecewise polynomial with the same chambers of polynomiality as the usual double Hurwitz numbers, thus giving a new proof of the piecewise polynomiality of (strictly) monotone double Hurwitz numbers. We further derive wall-crossing formulae for the $\lambda-$invariants in \cref{thm:wall} and a recursion in \cref{thm:rec}.

\subsection{Structure of this paper}
In \cref{sec:prelim}, we recall some of the basic facts around Hurwitz theory and tropical geometry. In \cref{sec:polwall}, we introduce the necessary notation to state two of our main results. Mainly, we state a piecewise polynomiality results in \cref{thm:poly} andd wall-crossing formulae for the aforementioned $\lambda-$invariants. In \cref{sec:proof}, we prove those theorems. Finally, we derive a recursion for $\lambda-$invariants in \cref{sec:rec}.

\subsection{Acknowledgements}
The authors are thankful to Hannah Markwig for many helpful correspondences and comments on an earlier draft. The first author gratefully acknowledges financial support as part of the LOEWE research unit 'Uniformized structures in Arithmetic and Geometry'. D.~L.~is supported by the Max Planck Gesellschaft.

\section{Preliminaries}
\label{sec:prelim}
In this section, we recall the basic background needed for this work. In particular, we introduce several variants of Hurwitz numbers in \cref{sec:hur}, review some basics of Gromov-Witten theory in \cref{sec:grom} and recall the tropical correspondence theorems expressing these variants in terms of tropical covers in \cref{sec:trop}. We further fix the notation $\zeta(z)=2\mathrm{sinh}(z/2)=e^{z/2}-e^{-z/2}$ and $\mathcal{S}(z)=\frac{\zeta(z)}{z}$.

\subsection{Hurwitz numbers}
\label{sec:hur}
We define monotone and strictly monotone Hurwitz numbers in terms of the symmetric group which we denote by $S_d$. For a permutation $\sigma\in S_d$, we denote the partition corresponding to its conjugacy class by $C(\sigma)$.

\begin{definition}
\label{def:mono}
Let $g$ be a non-negative integer, $x\in\left(\mathbb{Z}\backslash\{0\}\right)^n$ with $\sum x_i=0$.  Let $x^+$ (resp. $x^-$) be the tuple of positive entries of $x$ (resp. $-x$) and denote $d=|x^+|=|x^-|$. Further, we set $b=2g-2+n$. Then we define a factorisation of type $(g,x)$ to be a tuple $(\sigma_1,\tau_1,\dots,\tau_b,\sigma_2)$, such that
\begin{enumerate}
\item $\sigma_i,\tau_j\in S_d$;
\item $C(\sigma_1)=x^+$, $C(\sigma_2)=x^-$, $C(\tau_i)=(2,1,\dots,1)$;
\item $\sigma_2=\tau_b\cdots\tau_1\sigma_1$;
\end{enumerate}
Further, we denote $\tau_i=(r_i\,s_i)$ with $r_i<s_i$. We call $(\sigma_1,\tau_1,\dots,\tau_b,\sigma_2)$ a \textit{monotone} factorisation if $s_i\le s_{i+1}$ and strictly monotone if $s_i<s_{i+1}$. We then define the \textit{monotone double Hurwitz number} $h_{g;x}^{\le,\bullet}$ to be the number of monotone factorisations times $\frac{1}{d!}$. Analogously, we define the \textit{strictly monotone double Hurwitz number} by $h_{g;x}^{<,\bullet}$ to be the number of strictly monotone factorisations times $\frac{1}{d!}$.\par 
Furthermore, we call a factorisation of type $(g,x)$ transitive if
\begin{enumerate}
\item[(4)] $\langle\sigma_1,\sigma_2,\tau_1,\dots,\tau_b\rangle$ is a transitive subgroup of $S_d$.
\end{enumerate}
Then we define the \textit{connected monotone double Hurwitz number} $h_{g;x}^{\le,\circ}$ and the \textit{connected strictly monotone double Hurwitz number}  $h_{g;x}^{<,\circ}$ as before as the numbers of transitive (strictly) monotone factorisations of type $(g,x)$ times $\frac{1}{d!}$.
\end{definition}

\begin{remark}
By dropping the monotonicity condition on the transpositions in \cref{def:mono}, we obtain so-called \textit{double Hurwitz numbers}. These numbers are equivalent to the enumeration of branched degree $d$ morphisms $C\to\mathbb{P}^1_{\mathbb{C}}$ with ramification profile $x^+$ ($x^-$) over $0$ (resp. $\infty$) and simple ramification over $b$ fixed points of $\mathbb{P}^1_{\mathbb{C}}$.
\end{remark}

\subsection{Gromov-Witten invariants with target $\mathbb{P}^1$}
\label{sec:grom}
We now recall some of the notions of Gromov-Witten theory. A more detailed introduction in the context of tropical covers can be found in \cite{CJMRgraphical}. For a more general introduction to the topic, we recommend \cite{V}.\par
We denote by $\overline{\mathcal{M}}_{g,n}(\mathbb{P}^1,d)$ the moduli space of stable maps with $n$ marked points which a Deligne-Mumford stack of virtual dimension $2g-2+2d+n$. It consists of tuples $(X,x_1,\dots,x_n,f)$, such that $X$ is a connected, projective curve of genus $g$ with at worst nodal singularities, $x_1,\dots,x_n$ are non-singular points on $X$ and $f:X\to\mathbb{P}^1$ is a function with $f_{\ast}([X])=d[\mathbb{P}^1]$. Moreover, $f$ may only have a finite automorphism group (respecting markings and singularities). In order to define enumerative invariants, we introduce
\begin{itemize}
\item The $i-$th evaluation morphism is the map $ev_i:\overline{\mathcal{M}}_{g,n}(\mathbb{P}^1,d)\to\mathbb{P}^1$ obtained by mapping the tuple $(X,x_1,\dots,x_n,f)$ to $x_i$.
\item The $i-$th cotangent line bundle $\mathbb{L}_i\to\overline{\mathcal{M}}_{g,n}(\mathbb{P}^1,d)$ is obtained by identifying the fiber of each point with the cotangent space $\mathbb{T}^*_{x_i}(X)$. The first Chern class of $i-$th cotangent line bundle is called a psi class which we denote by $\psi_i=c_1(\mathbb{L}_i)$.
\end{itemize}

This yields the following definition.

\begin{definition}
Fix $g,n,d$ and let $k_1,\dots,k_n$ be non-negative integers, such that $k_1+\dots+k_n=2g+2d-2$. Then the stationary Gromov-Witten invariant is defined by
\begin{equation}
\langle\tau_{k_1}(pt)\cdots\tau_{k_n}(pt)\rangle_{g,n}^{\mathbb{P}^1}=\int_{[\overline{\mathcal{M}}_{g,n}(\mathbb{P}^1)]^{\vir}}\prod ev_i^*(pt)\psi_i^{k_i},
\end{equation}
where $pt$ denotes class of a point on $\mathbb{P}^1$.
\end{definition}

Similarly, we consider the moduli space of relative stable maps $\overline{\mathcal{M}}_{g,n}(\mathbb{P}^1,\nu,\mu,d)$ relative to two partitions $\mu,\nu$ of $d$ and define the relative Gromov-Witten invariants by

\begin{equation}
\label{eq:gwmunu}
\langle\nu\mid\tau_{k_1}(pt)\cdots\tau_{k_n}(pt)\mid\mu\rangle_{g,n}^{\mathbb{P}^1}=\int_{[\overline{\mathcal{M}}_{g,n}(\mathbb{P}^1,\nu,\mu,d)]^{\vir}}\prod ev_i^*(pt)\psi_i^{k_i}.
\end{equation}

We note that in the following, we add subscripts "$\circ$" and "$\bullet$" which correspond to \textit{connected} or \textit{not necessarily connected} (for simplicity also called \textit{disconnected}) Gromov-Witten invariants which in turn correspond to considering connected or disconnected stable maps.

\subsection{Tropical correspondence theorem}
\label{sec:trop}
We begin by defining abstract tropical curves.

\begin{definition}
\label{def:abstrop}
An \textit{abstract tropical curve} is a connected
metric graph $\Gamma$ with unbounded edges called ends, together with a function associating a genus $g(v)$ to each vertex $v$. Let $V(\Gamma)$ be the set of its vertices. Let $E(\Gamma)$ and $E'(\Gamma)$ be the set of its internal (or bounded) edges and its set of all edges, respectively. The set of ends is therefore $E'(\Gamma) \setminus E(\Gamma)$, and all ends are considered to have infinite length. The genus of an abstract tropical curve $\Gamma$ is
 $ g(\Gamma)\coloneqq h^1(\Gamma) + \sum_{v \in V(\Gamma)} g(v)$,
where $h^1(\Gamma)$ is the first Betti number of the underlying graph.
An \textit{isomorphism} of a tropical curve is an automorphism of the underlying graph that respects edges' lengths and vertices' genera.
The \textit{combinatorial type} of a tropical curve is obtained by disregarding its metric structure.
\end{definition}

As a next step, we consider maps between abstract tropical curves which mirror the situation of covers between Riemann surfaces.

\begin{definition}
\label{def:tropmorph}
A tropical cover is a surjective harmonic map $\pi:\Gamma_1\to\Gamma_2$ between abstract tropical curves as in \cite{ABBR}, i.e.:
\begin{itemize}
\item[\textit{i).}] Let $V(\Gamma_i)$ denote the vertex set of $\Gamma_i$, then we require $\pi(V(\Gamma_1))\subset V(\Gamma_2)$;
\item[\textit{ii).}] Let $E'(\Gamma_i)$ denote the edge set of $\Gamma_i$, then we require $\pi^{-1}(E'(\Gamma_2))\subset E'(\Gamma_1)$;
\item[\textit{iii).}] For each edge $e\in E'(\Gamma_i)$, denote by $l(e)$ its length. We interpret $e\in E'(\Gamma_1),\pi(e)\in E'(\Gamma_2)$ as intervals $[0,l(e)]$ and $[0,l(\pi(e))]$, then we require $\pi$ restricted to $e$ to be a linear map of slope $\omega(e)\in\mathbb{Z}_{\ge0}$, that is $\pi:[0,l(e)]\to[0,l(\pi(e))]$ is given by $\pi(t)=\omega(e)\cdot t$. We call $\omega(e)$ the \textit{weight} of $e$. If $\pi(e)$ is a vertex, we have $\omega(e)=0$.
\item[\textit{iv).}] For a vertex $v\in\Gamma_1$, let $v'=\pi(v)$. We choose an edge $e'$ adjacent to $v'$. We define the local degree at $v$ as
\begin{equation}
d_v=\sum_{\substack{e\in\Gamma_1\\\pi(e)=e'}}\omega_e.
\end{equation}
We require $d_v$ to be independent of the choice of edge $e'$ adjacent to $v'$. We call this fact the \textit{balancing} or \textit{harmonicity condition}.
\end{itemize}

We furthermore introduce the following notions:
\begin{itemize}
\item[\textit{i).}] The \textit{degree} of a tropical cover $\pi$ is the sum over all local degrees of pre-images of any point in $\Gamma_2$. Due to the harmonicity condition, this number is independent of the point in $\Gamma_2$.
\item[\textit{ii).}] For any end $e$, we define a partion $\mu_e$ as the partition of weights of the ends of $\Gamma_1$ mapping to $e$. We call $\mu_e$ the \textit{ramification profile} above $e$.
\end{itemize}
\end{definition}

The following theorem expresses monotone and strictly monotone double Hurwitz numbers in terms of tropical covers weighted by Gromov-Witten invariants.

\begin{theorem}[\cite{HL}]
\label{thm:trop}
Let $g$ be a non-negative integer, and  $x\in\left(\mathbb{Z}\backslash\{0\}\right)^n$ wih $|x^+|=|x^-|=d$. 

\begin{align}
h_{g; x}^{\leq,\bullet}&=\sum_{\lambda\vdash b}
\sum_{\pi \in \Gamma( \mathbb{P}^1_{\text{trop}}, g; x,\lambda)}\frac{1}{|\mathrm{Aut}(\pi)|}\frac{1}{\ell(\lambda)!}\prod_{v \in V(\Gamma)} m_v \prod_{e \in E(\Gamma)} \omega_e
\\
h_{g; x}^{<,\bullet}&=\sum_{\lambda\vdash b}
\sum_{\pi \in \Gamma( \mathbb{P}^1_{\text{trop}}, g; x,\lambda)}\frac{1}{|\mathrm{Aut}(\pi)|}\frac{1}{\ell(\lambda)!}\prod_{v \in V(\Gamma)} (-1)^{1 + \mathrm{val}(v)} m_v \prod_{e \in E(\Gamma)} \omega_e
\end{align}
where $\Gamma( \mathbb{P}^1_{\text{trop}}, g; x,\lambda)$ is the set of tropical covers 
$
\pi: \Gamma \longrightarrow \mathbb{P}^1_{trop} = \mathbb{R}
$
with $b=2g-2+n$ points $p_1,\dots,p_b$ fixed on the codomain $\mathbb{P}^1_{trop}$ and $\lambda$ an ordered partition of $b$, such that
\begin{itemize}
\item[\textit{i).}] The unbounded left (resp. right) pointing ends of $\Gamma$ have weights given by the partition $x^+$ (resp. $x^-$).
\item[\textit{ii).}] The graph $\Gamma$ has $l:=\ell(\lambda)\le b$ vertices. Let $V(\Gamma) = \{v_1, \dots, v_l\}$ be the set of its vertices. Then we have $\pi(v_i)=p_i$ for $i=1,\dots,l$. Moreover, let $w_i  = \mathrm{val}(v_i)$ be the corresponding valencies.
% such that 
%$$
%\sum_{i=1}^l w_i  =  \ell(\mu) + \ell(\nu) + l - 2
%$$
\item[\textit{iii).}] We assign an integer $g(v_i)$ as the genus to $v_i$ and the following condition holds true

\begin{equation}
h^1(\Gamma) + \sum_{i=1}^l  g(v_i) = g.
\end{equation}
\item[\textit{iv).}] We have $\lambda_i=\mathrm{val}(v_i)+2g(v_i)-2$.
\item[\textit{v).}] For each vertex $v_i$, let $y^{+}$ (resp. $ y^-$) be the tuple of weights of those edges adjacent to $v_i$ which map to the right-hand (resp. left-hand) of $p_i$. The multiplicity $m_{v_i}$ of $v_i$ is defined to be
\begin{align}
m_{v_i} = &(\lambda_i-1)!|\mathrm{Aut}(y^{+})||\mathrm{Aut}(y^{-})|\\
&\sum_{g_1^i+g_2^i=g(v_i)}\cor{ \!\! \tau_{2g^i_2 - 2}(\omega) \!\! }_{g^i_2}^{\mathbb{P}^1, \circ} \cor{ \!\!y^+,   \tau_{2g^i_1 - 2 + n}(\omega) , y^-\!\!  }^{\mathbb{P}^1,\circ}_{g^i_1}
\end{align}
\end{itemize}
Furthermore, we obtain $h_{g;x}^{\le,\circ}$ and  $h_{g;x}^{<,\circ}$ by considering only connected source curves.
\end{theorem}

In the following remark, we discuss the Gromov-Witten invariants appearing in the above vertex multiplicities.

\begin{remark}
It is well-known that
\begin{equation}
\left \langle \tau_{2l-2}(\omega) \right \rangle_{l,1}^{\mathbb{P}^1}=[z^{2l-1}]\frac{1}{\zeta(z)}=-\frac{2^{2l-1}-1}{2^{2l-1}}\frac{B_{2l}}{(2l)!},
\end{equation}
where $B_{2l}$ is the $2l-$th Bernoulli number. Furthermore, it was proved in \cite{OP2} that
\begin{equation}
\left\langle{y}^+,   \tau_{2g - 2 + \ell({y}^+) + \ell({y}^-)} , y^-  \right\rangle^{\mathbb{P}^1, \circ}_g
=
\frac{1}{|\Aut ({y}^+)||\Aut ({y}^-)|}[z^{2g}] \frac{\prod_{{y}^+_i} \mathcal{S}({y}_i z)\prod_{{x}^-_i} \mathcal{S}({x}_i z)}{\mathcal{S}(z)}.
\end{equation}
\end{remark}

\section{Piecewise polynomiality and Wall-crossings}
\label{sec:polwall}
We begin by defining a refinement of monotone and strictly monotone double Hurwitz numbers.

\begin{definition}
Let $g$ be a non-negative integer $x\in\mathbb{Z}^n$, such that $\left|x^+\right|=\left|x^-\right|$. Furthermore, let $\lambda'$ be an ordered partition of $2g-2+n$. Then we define
\begin{equation}
\vec{h}_{g;x,\lambda'}^{\le,\circ}=\sum_{\pi \in \Gamma( \mathbb{P}^1_{\text{trop}}, g; x,\lambda')}\frac{1}{|\mathrm{Aut}(\pi)|}\frac{1}{\ell(\lambda)!}\prod_{v \in V(\Gamma)} m_v \prod_{e \in E(\Gamma)} \omega_e.
\end{equation}

Furthermore, let $\lambda''$ be an unordered partition of $2g-2+n$. Then we define

\begin{equation}
h_{g;x,\lambda''}^{\le,\circ}=\sum_{\lambda}\vec{h}_{g;x,\lambda}^{\le,\circ}\,,
\end{equation}
where the first sum is over all ordered partitions $\lambda$ which are obtained by some ordering of $\lambda''$. Similarly, we define $\vec{h}_{g;x,\lambda'}^{<,\circ}$ and $h_{g;x,\lambda'}^{<,\circ}$. We further define their disconnected counterparts by considering disconnected tropical covers and decorate them with $\bullet$.
\end{definition} 

\begin{remark}
We observe that by definition
\begin{equation}
\label{equ:finer}
h_{g;x}^{\le,\circ}=\sum_{\lambda'}\vec{h}_{g;x,\lambda'}^{\le,\circ}=\sum_{\lambda''}h_{g;x,\lambda''}^{\le,\circ},
\end{equation}
where the first sum is taken over all ordered partition $\lambda'$ of $2g-2+n$ and the second sum is taken over all unordered partitions $\lambda''$ of $2g-2+n$.\par 
We note that these numbers naturally appear as weighted sums of vacuum expectations of products of the $\mathcal{G}_l$ operators in the notation of \cite{HL}.
\end{remark}

\subsection{Results}
In this section, we collect our results about the piecewise polynomial behaviour of $h_{g;x,\lambda}^{\le,\circ}$ and $h_{g;x,\lambda}^{<,\circ}$. We first define the \textit{resonance arrangement} which is the hyperplane arrangement in $\mathbb{R}^n$ given by
\begin{equation}
W_I=\left\{x\in\mathbb{Z}^n\mid\sum_{i\in I}x_i=0\right\}
\end{equation}
for all $I\subset\{1,\dots,n\}$. The connected components of the complement of the resonance arrangements are called \textit{chambers}. We also refer to them by $H-$\textit{chambers}.

\begin{theorem}
\label{thm:poly}
Let $g$ be a non-negative integer, fix the length $n$ of $x$ and let $\lambda$ be an unordered partition of $2g-2+n$. The function $h_{g;x,\lambda}^{\le,\circ}$ and $h_{g;x,\lambda}^{<,\circ}$ are polynomials of degree at most $4g-3+n$ in each chamber of the resonance arrangement.
\end{theorem}

Combining \cref{thm:poly} and \cref{equ:finer}, we therefore obtain a new proof of the following result.

\begin{corollary}[{\cite{goulden2016toda,HKL}}]
For a non-negative integer $g$ and a fixed length $n$ of $x$, the functions $h_{g;x}^{\le,\circ}$ and $h_{g;x}^{<,\circ}$ are piecewise polynomial.
\end{corollary}

This motivates the following definition.

\begin{definition}
Let $\mathfrak{c}_1,\mathfrak{c}_2$ be two $H-$\textit{chambers} adjacent along the wall $W_I$, with $\mathfrak{c}_1$ being the chamber with $x_I=\sum_{i\in I} x_i<0$. Let $P_i^{\lambda}(x)$ be the polynomial expressing $h_{g;x,\lambda}$ in $\mathfrak{c}_i$. We define the wall-crossing function by
\begin{equation}
WC_I^{\lambda}(x)=P^{\lambda}_2(x)-P^{\lambda}_1(x).
\end{equation}
\end{definition}

We derive the following expression of the wall-crossing function.

\begin{theorem}
\label{thm:wall}
Let $g$ be a non-negative integer, $n$ the fixed length of $x$ and $\lambda$ an unordered partition of $b=2g-2+n$. Then we have
\begin{align}
WC_I^{\lambda}(x)=\sum_{\substack{|y|=|z|=|x_I|}}\sum_{\substack{\lambda^i\textrm{unordered}\\\lambda^1\cup\lambda^2\cup\lambda^3=\lambda}}
&
\left((-1)^{\ell(\lambda^2)}\frac{\prod y_i}{\ell(y)!}\frac{\prod z_i}{\ell(z)!} h^{\le,\circ}_{g_1;(x_I,-y),\lambda^1}h^{\le,\bullet}_{g_2;(y,-z),\lambda^2}h^{\le,\circ}_{g_3;(z,x_{I^c}),\lambda^3}\right),
\end{align}
where $y$ (resp. $z$) is an ordered tuple of length $\ell(y)$ (resp. $\ell(z)$) of positive integers with sum $|y|$ (resp. $|z|$) and $g_1$ is given by $|\lambda^1|=2g_1-2+\ell((x_I,-y))$ (and analogously for $g_2,g_3)$).
\end{theorem}

\section{Proofs of chamber polynomiality and of wall crossing formulae}
\label{sec:proof}
In this section, we prove \cref{thm:poly} and \cref{thm:wall}. We focus on the case of monotone Hurwitz numbers as the other case is completely parallel. To begin with, we introduce a formal set-up for the proofs of both theorems in \cref{sec:setup}. We continue in \cref{sec:poly} where we prove \cref{thm:poly}. Finally, we prove \cref{thm:wall} in \cref{sec:wall}. We follow the strategy of \cite{CJMa} which focuses on the case of trivalent graphs, however all results we cite hold for the graphs with higher valency considered in this paper with the same proofs. We also provide a running example for this case of higher valency throughout the proof, which is analogous to example 2.5 in \cite{CJMa} for the trivalent case.

\subsection{Formal set-up}
\label{sec:setup}
Instead of tropical covers, we work with combinatorial covers, where the information given by the cover is encoded as an orientation given on the graph.

\begin{definition}[Combinatorial cover] \label{def:monograph}
For fixed $g$, $x=(x_1,\dots,x_n)\in(\mathbb{Z}\backslash\{0\})^n$, $\lambda\vdash2g-2+n$ unordered, a graph $\Gamma$ is a \textit{combinatorial cover} of type $(g,x,\lambda)$, if
\begin{enumerate}
\item $\Gamma$ is a connected graph with at most $2g-2+2n$ vertices;
\item $\Gamma$ has $n$ many $1-$valent vertices called \textit{leaves}; the adjacent edges are called \textit{ends} and are labeled by the weights $x_1,\dots,x_n$; further, all ends are oriented inwards. If $x_i>0$, we say it is an \textit{in-end}, otherwise it is an out-end;
\item we denote the set of edges which are \textit{not} edges by $E^{in}(\Gamma)$;
\item there are $\ell(\lambda)$ inner vertices;
\item we denote the inner vertices by $v_1,\dots,v_{\ell(\lambda)}$ and assign a non-negative integer $g(v_i)$ to $v_i$ which we call the \textit{genus of $v_i$}; we further have $\lambda_i=\val(v_i)+2g(v_i)-2$;
\item after reversing the orientation of the \textit{out-ends}, $\Gamma$ does not have sinks or sources;
\item the internal vertices are ordered compatibly with the partial ordering induced by the directions of the edges;
\item we have $g=b_1(\Gamma)+\sum g(v_i)$, where $b_1(\Gamma)$ is the first Betti number of $\Gamma$;
\item every internal edge $e$ of the graph is equipped with a weight $\omega(e)\in\mathbb{N}$. The weights satisfy the balancing condition a each inner vertex: the sum of all weights of incoming edges equals the sum of the weights of outgoing edges.
\end{enumerate}
The notation $\Gamma(x,\lambda,d,o)$ indicates that graph comes with directed edges $(d)$ and with a compatible vertex ordering $(o)$.
\end{definition}

Then \cref{thm:trop} translates to
\begin{equation}
\label{equ:hur}
h_{g; x}^{\leq, \circ}=\sum _{\lambda\vdash b}\sum_{\Gamma}\frac{1}{|\mathrm{Aut}(\Gamma)|}\frac{1}{\ell(\lambda)!}\varphi_{\Gamma},
\end{equation}
where the second sum is over all combinatorial covers $\Gamma$ of type $(g,x,\lambda)$ and we have
\begin{equation}
\varphi_{\Gamma}=\prod_{i=1}^{\ell(\lambda)}m_{v_i}\prod_{e\in E^{in}(\Gamma)}\omega(e)
\end{equation}
with 
\begin{align}
m_{v_i} = (\lambda_i-1)!|\mathrm{Aut}(y^{+})||\mathrm{Aut}(y^{-})|\sum_{g_1^i+g_2^i=g(v_i)}\cor{ \!\! \tau_{2g^i_2 - 2}(\omega) \!\! }_{g^i_2}^{\mathbb{P}^1, \circ} \cor{ \!\! y^+,   \tau_{2g^i_1 - 2 + \ell(y^+) + \ell(\textbf{y}^-)}(\omega) , y^-\!\!  }^{\mathbb{P}^1,\circ}_{g^i_1}
\end{align}
where $y^+$ is the tuple of weights of in-coming edges and $y^-$ the tuple of weights of outgoing edges at $v_i$. Analogously, one obtains $h_{g;x,\lambda}^{\le,\circ}$, $\vec{h}_{g;x,\lambda}^{\le,\circ}$ and their disconnected counterparts.\par 
Moreover, for an unordered partition $\lambda$, we have
\begin{equation}
h_{g; x,\lambda}^{\leq, \circ}=\sum_{\Gamma}\frac{1}{|\mathrm{Aut}(\Gamma)|}\frac{1}{\ell(\lambda)!}\varphi_{\Gamma}.
\end{equation}
where the second sum is over all combinatorial covers $\Gamma$ of type $(g,x,\lambda)$.

\begin{definition}
Given $g$ and $x$, an $x-$graph $\Gamma(x)$ (or simply $\Gamma$ when there is no risk of confusion) is a connected, genus $g$ combinatorial cover, where we forget the direction of the edges and the vertex ordering, such that the $n$ ends are labeled $x_1,\dots,x_n$.
\end{definition}

\subsubsection{Hyperplane arrangements}
We view an $x-$graph $\Gamma$ as a one-dimensional cell complex. The differential $d:\R E_{\Gamma}\to\R V_{\Gamma}$, sending a directed edge to the difference of its head and tail vertices, yields the following short exact sequence
\begin{equation}
0\to\mathrm{ker}(d)\to\R E_{\Gamma}\to \mathrm{im}(d)\to 0.
\end{equation}
We decompose $\R E_{\Gamma}=\R^n\bigoplus\R^{|E^{in}(\Gamma)|}$ into ends and internal vertices. Then we have a vector of the form $(x,0)\in\mathrm{im}(d)$ when $\sum x_i=0$.

\begin{definition}
We define the \textit{space of flows} to be
\begin{equation}
F_{\Gamma}(x)=d^{-1}(x,0).
\end{equation}
Inside the space of flows, we define a hyperplane arrangement
\begin{equation}
\mathcal{A}_{\Gamma}(x)
\end{equation}
given by the restriction of the coordinate hyperplanes corresponding to the internal edges in $\R E_{\Gamma}$. The defining polynomial for this hyperplane arrangement is
\begin{equation}
\varphi_{\mathcal{A}}=\prod e_i,
\end{equation}
where $e_i$ are the coordinate functions on $\R E_{\Gamma}$ restricted to $F_{\Gamma}(x)$.
\end{definition}

We note that often it is useful to fix a reference orientation on a given $x-$graph. The following lemma shows that this corresponds to fixing a bounded chamber in the hyperplane arrangement.

\begin{lemma}[{\cite[Lemma 2.13, Corollary 2.14]{CJMa}}]
The bounded chambers of $\mathcal{A}_{\Gamma}(x)$ correspond to orientations of $\Gamma$ with no directed cycles. Moreover, given an $(x,\lambda)-$ graph $\Gamma$, the bounded chambers of $\mathcal{A}_{\Gamma}(x)$ are in bijection with directed $(x,\lambda)-$graphs projecting to $\Gamma$ after forgetting the orientations of the edges that come from a combinatorial cover (defined in \ref{def:monograph}).  \qed
\end{lemma}

The following remark indicates an interesting structural result regarding the vertex contributions.

\begin{remark}
\label{rem:vert}
Recall that the contribution of each vertex is given by
\begin{align}
m_{v_i} = (\lambda_i-1)!|\mathrm{Aut}(y^{+})||\mathrm{Aut}(y^{-})|\sum_{g_1^i+g_2^i=g(v_i)}\cor{ \!\! \tau_{2g^i_2 - 2}(\omega) \!\! }_{g^i_2}^{\mathbb{P}^1, \circ} \cor{ \!\!y^+,   \tau_{2g^i_1 - 2 + \ell(y^+) + \ell(y^-)}(\omega) , y^-\!\!  }^{\mathbb{P}^1,\circ}_{g^i_1},
\end{align}
where $y^+$ are the incoming and $y^-$ are the outgoing edge weights. Moreover, by \cite[Theorem 2]{OP} the following identity holds
\begin{equation}
\cor{ \!\!y^+,   \tau_{2g^i_1 - 2 + \ell(y^+) + \ell(y^-)}(\omega) , y^-\!\!  }^{\mathbb{P}^1,\circ}_{g^i_1}=\frac{1}{|\mathrm{Aut}(y^{+})|}\frac{1}{|\mathrm{Aut}(y^{-})|}[w^{g_1^i}]\frac{\prod_{y^+}\mathcal{S}(y^+_iw)\prod_{y^-}\mathcal{S}(y^-_iw)}{\mathcal{S}(w)}.
\end{equation}
Thus we obtain
\begin{align}
m_{v_i} = (\lambda_i-1)!\sum_{g_1^i+g_2^i=g(v_i)}\cor{ \!\! \tau_{2g^i_2 - 2}(\omega) \!\! }_{g^i_2}^{\mathbb{P}^1, \circ} [w^{g_1^i}]\frac{\prod_{y^+}\mathcal{S}(y^+_iw)\prod_{y^-}\mathcal{S}(y^-_iw)}{\mathcal{S}(w)}.
\end{align}
We recall that $\mathcal{S}(w)=1+\frac{z^2}{24}+\frac{z^4}{1920}+O(z^6)$ and $\frac{1}{\mathcal{S}(w)}=1-\frac{z^2}{24}+\frac{7z^4}{5760}+O(z^6)$ are even power series. Therefore $m_{v_i}$ is a polynomial in the adjacent edge weights and all appearing monomials are of even degree. We denote this polynomial by $M(v_i)$. This polynomial is independent of the flow of the respective branching graph.
\end{remark}

\begin{definition}
Let $\Gamma$ be an $x-$graph. We denote by $S_{\Gamma}(x)$ the contribution to $h_{g;x,\lambda}$ of all combinatorial covers having underlying $(x)-$graph $\Gamma$, where $\lambda$ is obtained by 
$$
\lambda_i=\mathrm{val}(v_i)+2g(v_i)-2,
$$
where $v_i$ runs over all inner vertices, i.e. $\lambda=(\lambda_1,\dots,\lambda_{\ell(\lambda)})$.\par 
For a given $(x)-$graph $\Gamma$, we call $F-$chambers the chambers of $\mathcal{A}_{\Gamma}(x)$ in the flow space $F_{\Gamma}(x)$. Recall that all points in the same $F$-chamber $A$ have edge weights with the same sign (i.e. their edges have the same orientation). Crossing a wall towards a different chamber in a certain direction means moving in the flow space $F_{\Gamma}(x)$ along the direction $e_i \rightarrow 0 $ (say, in the chamber $A$ we have $e_i >0$ ), where $e_i$ is the coordinate that represents the weight of some edge in the decomposition $\R V_{\Gamma}=\R^n\bigoplus\R^{|E(\Gamma)|}$.
After hitting the wall defined by $e_i = 0$, the adjacent chamber has all coordinates $e_j$ with same sign as in the chamber $A$, for $j \neq i$, and instead $e_i <0$. Each point of this chamber corresponds therefore to an oriented graph in which the edge corresponding to $e_i$ has opposite orientation with respect to the one in chamber $A$.\par
For an $F-$chamber $A$, let $\Gamma_A$ denote the directed $(x,\lambda)$-graph with the edge directions corresponding to the chamber $A$. 
We use $m(A)$ (or $m(\Gamma_A)$), to denote the number of all possible orderings of the vertices of $\Gamma_A$ from left to right (recall that the branch points are fixed over the base).
\begin{lemma}\cite{CJMa} For an $F$-chamber $A$, we have that $m(A)$ is zero if and only if $A$ is unbounded.
\end{lemma}
Roughly speaking, the reason for the above statement is that a chamber $A$ can be unbounded if and only if the graphs contain an oriented loop which makes it impossible to order the vertices over the base. As $m(A)$ will appear as multiplicity in our formula, we can immediately discard all unbounded chambers, as their contribution vanishes completely.\par
We use $\mathrm{Ch}(\mathcal{A}_{\Gamma}(x))$ to denote the set of $F$-chambers of $\mathcal{A}_{\Gamma}(x)$.  
Clearly, the sign of 
$$
\varphi_{\mathcal{A}} = \prod^{n + |E(\Gamma)|}_{i=1} e_i
$$ 
alternates on adjacent $F-$chambers (since we swap the direction of one edge, as explained above): we indicate with
 $\mathrm{sign}(A) = (-1)^{N(A)}$
  the sign of $\varphi_{\mathcal{A}}$ on the chamber $A$, where $N(A)$ is the number of negative coordinates $e_i$ in the chamber $A$.\par 
For integer values of $x$, the space of flows $F_{\Gamma}(x)$ has an affine lattice, coming from the integral structure of $\Z E_{\Gamma}$. We denote this lattice by
\begin{equation}
\Lambda=F_{\Gamma}(x)\cap\Z E_{\Gamma}.
\end{equation}
This notation allows a convenient interpretation of $S_{\Gamma}(x)$ in terms of the hyperplane arrangement $\mathcal{A}_{\Gamma}(x)$. Choices of the weights of the edges -- i.e. the choice of a flow $f$ on $\Gamma$ - correspond to lattice points in $\Lambda$. We have that
\begin{align}
\label{equ:smult}
S_{\Gamma}(x)&= \frac{1}{|\mathrm{Aut}(\Gamma)|}\sum_{A\in\mathrm{Ch}(\mathcal{A}_{\Gamma}(x))} m(A) \sum_{f\in A\cap\Lambda}\left(\prod_{e \in E'(\Gamma)} w(e) \prod_iM(v_i)\right),
\\
\nonumber
&=\frac{1}{|\mathrm{Aut}(\Gamma)|}\sum_{A\in\mathrm{Ch}(\mathcal{A}_{\Gamma}(x))}\mathrm{sign}(A)\, m(A)\sum_{f\in A\cap\Lambda}\left(\varphi_{\mathcal{A}}(f)\prod_iM(v_i)\right),
\end{align}
 where $\prod_iM(v_i)$ is an even polynomial in the edge weights by \cref{rem:vert}, and to pass from the first to the second line use that the product of all the edge weights of a flow $f$ is the absolute value of $\varphi_{\mathcal{A}}$ computed at $f = (e_i)_i$ which if $f\in A$ is simply $\mathrm{sign}(A)\varphi_{\mathcal{A}}(f)$.
\end{definition}

\begin{example}
We illustrate the introduced notions for the combinatorial cover $\Gamma(x,\lambda,d,o)$ in the top of \cref{ex:flsp}, where $o$ is indicated in the left picture, $d$ is indicated by the directed edges, $\lambda=(1,1,1,1,5,2)$ and $x=(x_1,\dots,x_5)$.\par 
In the middle of \cref{ex:flsp}, two flow spaces for $\Gamma$ are given. On the left, we have $-(x_4+x_5)>x_2$ and $x_1+x_3>0$. On the right, we have crossed the wall $x_1+x_3=0$.\par 
We further have $M(v_i)=1$ for $i\neq 5$ and
\begin{equation}
M(v_5)=\frac{3a^4+10a^2(b^2+c^2)+3b^4+10b^2c^2+3c^4}{5760}
\end{equation}
for $a=i, b=-j-(x_4+x_5),c=-i-j-(x_4+x_5)$.
\end{example}

\begin{figure}
\scalebox{0.6}{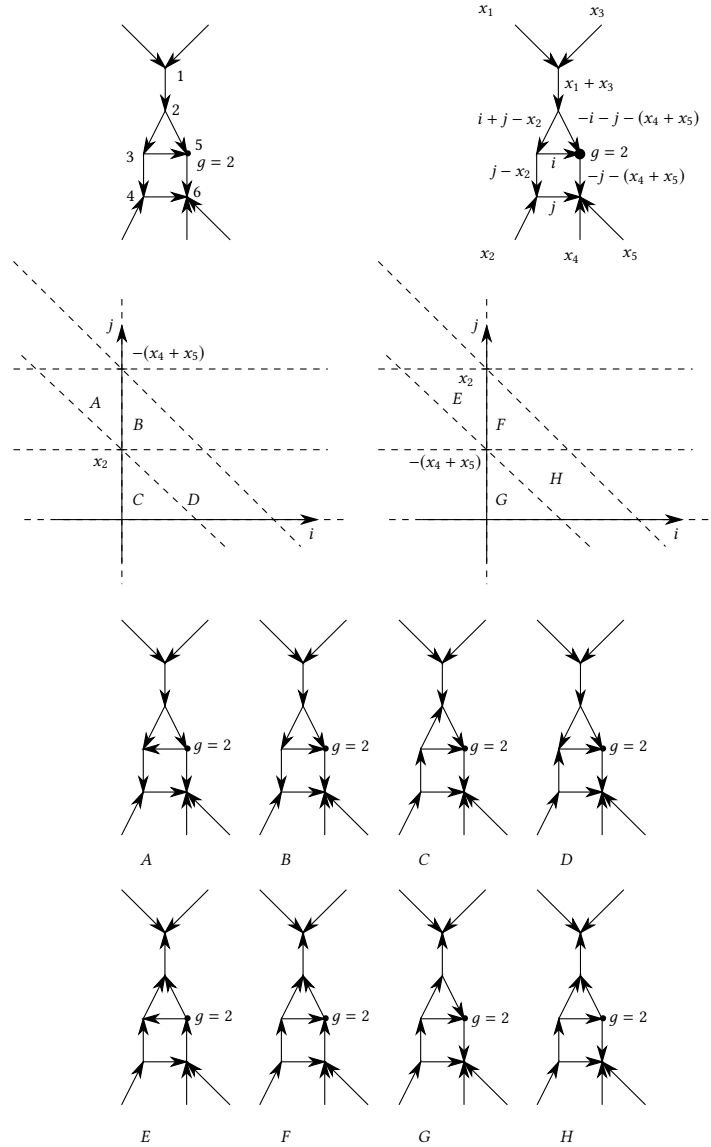}
\caption{A combinatorial cover, its flow space in two adjacent chambers and the corresponding orientations.}
\label{ex:flsp}
\end{figure}

\subsection{Polynomials and walls}
\label{sec:poly}
We begin with the proof of \cref{thm:poly}. We fix an $(x)-$graph $\Gamma$ with reference orientation given by the flow $f$. We first observe that
\begin{equation}
\frac{1}{|\mathrm{Aut}(\Gamma)|}\left(\varphi_{\mathcal{A}}(f)\prod_iM(v_i)\right)
\end{equation}
is a polynomial of degree $|E(\Gamma)|+2\sum g_i$, as $\varphi_{\mathcal{A}(f)}$ is a polynomial of degree $|E(\Gamma)|$ and $M(v_i)$ is a polynomial of degree $2g_i$. Considering the Euler characteristic of $\Gamma$ we obtain
\begin{equation}
|E(\Gamma)|=\ell(\lambda)+b_1(\Gamma)-1=\ell(\lambda)+g-\sum g_i-1
\end{equation}
and therefore
\begin{equation}
|E(\Gamma)|+2\sum g_i=\ell(\lambda)+g+\sum g_i-1.
\end{equation}
Recalling $\lambda_i=\mathrm{val}(v_i)+2g(v_i)-2$ and the fact that $\mathrm{val}(v_i)\ge 2$, it is easily seen that the right hand side maximizes for $\lambda=(1,\dots,1)$. Thus, we have
\begin{equation}
|E(\Gamma)|+2\sum g_i\le 3g-3+n.
\end{equation}
Similar to \cite[Remark 2.11]{CJMa}, we have that $F_{\Gamma}(x)$ is $b_1(\Gamma)-$dimensional.\par
Moreover, it is well-known that summing a polynomial of degree $d$ over the lattice points in a $b_1(\Gamma)-$dimensional integral polytope of fixed topology is a polynomial of degree $d+b_1(\Gamma)$ in the numbers defining the boundary of the polytope. We further observe that each vertex is given by an integer vector because the incidence matrix of a directed graph is totally unimodular.\par 
Combining these facts, it follows that $S_{\Gamma}(x)$ is a polynomial in $x$ of degree $\ell(\lambda)+g+\sum g_i-1+b_1(\Gamma)$ as long as varying $x$ does not change the topology of $\mathcal{A}_{\Gamma}(x)$ which is maximal for $\ell(\lambda)=b$ and $g_i=0$.\par 
Thus $h_{g;x,\lambda}$ is piecewise polynomial of maximal degree $4g-3+n$. We now determine the areas in which $h_{g;x,\lambda}$ is polynomial. More precisely, we prove that $h_{g;x,\lambda}$ is polynomial in each top-dimensional component in the complement of a hyperplane arrangement. We further compute those hyperplanes.\par 
We note that the hyperplane arrangement given by $\mathcal{A}_{\Gamma}(x)$ is not always given by hyperplanes which only intersect transversally. Morally, the shape of the polynomial expressing $h_{g;x,\lambda}$ should only change when the topology of $\mathcal{A}_{\Gamma}(x)$ changes. When translating generic hyperplane arrangements the topology changes when one passes through a non-transversality. However, in our situation, there can be nontransversalities which appear for each value of $x$. Nonetheless, it is still true that the topology changes once one passes through additional nontransversalities. We call those nontransversalities which appear for any value $x$ \textit{good transversalities}. The following definition is a classification of these.

\begin{definition}
Suppose a set of $k$ hyperplanes (equivalently, edges in $\Gamma$) in $\mathcal{A}_{\Gamma}(x)$ intersect in codimension $k-l$. We call this intersection \textit{good} if there is a set $L$ of $l$ vertices in $\Gamma$ so that $I$ is precisely the set of edges incident to vertices in $L$.\par 
Furthermore, we define the \textit{discriminant locus} $\mathcal{D}\subset \R^n$ the set of values of $x$ so that for some directed $(x)-$ graph $\Gamma$ the hyperplane arrangement $\mathcal{A}_{\Gamma}(x)$ has a nontransverse intersection that is not good. The discriminant is a union of hyperplanes which we call the \textit{discriminant arrangement}. We call these hyperplanes \textit{walls} and the chambers defined by the arrangement $H-$\textit{chambers}.
\end{definition}

The $H-$chambers are the chambers of polynomiality of $h_{g;x,\lambda}$. Now, we establish that the walls correspond to the resonance arrangements
\begin{equation}
\sum_{i\in I}x_i=0
\end{equation}
for $I\subset\{1,\dots,n\}$. We begin with the following definition.

\begin{definition}
A \textit{simple cut} of a graph $\Gamma$ is a minimal set $C$ of edges that disconnects the ends of $\Gamma$: There are two ends of $\Gamma$ such that every path between them contains an edge of $C$ and this is true of no proper subset of $C$.\par 
For an $(x,\lambda)-$graph, a flow in $F_{\Gamma}(x)$ is \textit{disconnected} if for some simple cut $C$ the flow on each edge of $C$ is zero.
\end{definition}

This yields the following lemma.

\begin{lemma}[{\cite[Lemma 3.8]{CJMa}}]
The discriminant arrangement $\mathcal{D}$ is given by the set of $x\in\R^n$ such that for some $x-$graph $\Gamma$, the space $F_{\Gamma}(x)$ admits a disconnected flow.
\end{lemma}

Now, let $\Gamma$ admit a disconnected flow and let $C$ be the corresponding simple cut. Then it follows by the balancing condition that the sum $\sum_{i\in I}x_i$ of weights of ends belonging to a connected component of $\Gamma\backslash C$ is $0$. Thus, the walls of the discriminant arrangement are a subset of the hyperplanes in the resonance arrangement. The arrangements are equal since it is easy to construct a graph $\Gamma$, with some edge $e$, such that $\Gamma\backslash e$ has two components, one containing the ends of $I$ and the other containing the ends of $I^c$. Thus $h_{g;x,\lambda}$ is polynomial in each chamber of the resonance arrangement.

\subsection{Wall-crossing}
\label{sec:wall}
In this section, we prove \cref{thm:wall}. We first discuss the combinatorics of cutting an $x-$graph $\Gamma$ into several smaller graphs.

\begin{definition}
Let $\Gamma$ a directed graph and $E$ a subset of the edges of $\Gamma$. We consider the graph whose edges are the connected components of $\Gamma\backslash E^c$ and whose vertices are $E^C$. We call this graph the \textit{contraction} of $\Gamma$ with respect to $E$ and denote it by $\faktor{\Gamma}{E}$.\par 
We fix a directed $x-$graph $\Gamma_A$ and let $I\subset\{1,\dots,n\}$ some subset. Then the \textit{set $\mathrm{Cuts}_I(\Gamma_A)$ of $I-$cuts of $\Gamma_A$} consists of those subset $C$ of the edges of $\Gamma$, such that $C=\emptyset$ or
\begin{enumerate}
\item $\Gamma_A\backslash C$ is disconnected;
\item the ends of $\Gamma_A$ lie on exactly two components of $\Gamma_A\backslash C$, one containing all ends indexed by $I$, the other containing all ends indexed by $I^c$;
\item the directed graph $\faktor{\Gamma}{C^c}$ is \textit{acyclic} and has the component containing $I$ as the initial vertex and the component containing $I^c$ as the final vertex.
\end{enumerate}
Let $v(\Gamma_A\backslash C)$ be the number of components of $\Gamma_A\backslash C$. Then we define the \textit{rank of $C$} by
\begin{equation}
\mathrm{rk}(C)=v(\Gamma_A\backslash C)-1.
\end{equation}
\end{definition}

By the discussion in \cite[Section 6]{CJMa}, we have
\begin{equation}
\label{equ:heavy}
WC(x_2)=\sum_{\Gamma}\sum_{A\in \mathcal{BC}_{\Gamma}(x_2)}\sum_{C\in\mathrm{Cut}_I(\Gamma_A)}(-1)^{\mathrm{rk}(C)-1}\binom{\ell(\lambda)}{s,t_1,\dots,t_N,u}\left(\sum_{\Lambda\cap A}\varphi_{\mathfrak{A}}\prod_iM(v_i)\right),
\end{equation}
where $N=\mathrm{rk}(C)-1$ and $t_1,\dots,t_N$ are the numbers of inner vertices of the $N$ inner components of $\Gamma_A\backslash C$.\par 

\begin{definition}
Let $\Gamma$ be an $x-$graph and $C\in\mathrm{Cuts}_I(\Gamma_A)$. We call $C$ a thin cut, if if all edges in $C$ are either adjacent to the inital component containing $I$ or the component containg $I^c$. Furthermore, for a thin cut $T$, we denote by $P(T)$ the set of all cuts $C\in\mathrm{Cuts}_I(\Gamma_A)$ which contain $T$.
\end{definition}

By \cite[Lemma 8.2]{CJMa}, we have

\begin{equation}
\label{equ:inex}
(-1)^{t}\binom{\ell(\lambda)}{s,t,u}=\sum_{C\in \mathrm{P}(T)}(-1)^{\mathrm{rk}(C)-1}\binom{\ell(\lambda)}{s,t_1,\dots,t_N,u}.
\end{equation}

\begin{remark}
We note that there is a sign mistake in the formulation of \cite[Lemma 8.2]{CJMa} which occurs in the proof of \cite[Lemma 8.4]{CJMa}.
\end{remark}

Combining \cref{equ:heavy} and \cref{equ:inex}, we obtain

\begin{equation}
WC(x_2)=\sum_{\Gamma}\sum_{A\in \mathcal{BC}_{\Gamma}(x_2)}\sum_{\substack{T\in\mathrm{Cut}_I(\Gamma_A)\\\textrm{thin}}}(-1)^{t}\binom{\ell(\lambda)}{s,t,u}\left(\sum_{\Lambda\cap A}\varphi_{\mathcal{A}}\prod_iM(v_i)\right).
\end{equation}

We now observe that each thin cut divides $\Gamma_A$ into three parts: the initial component $\Gamma_A^1$, an intermediate part $\Gamma_A^2$ and a final component $\Gamma_A^3$. Moreover, the intermediate part may be disconnected. Thus, we observe that $\Gamma_A^1$ contributes to $h_{g_1;(x_I,-y),\lambda_1}^{\le}$, $\Gamma_A^2$ to $h_{g_2;(y,-z),\lambda_2}^{\le,\bullet}$ and $\Gamma_A^3$ to $h_{g_3;(z,x_{I^c})}^{\le,\circ}$, $\lambda_1\cup\lambda_2\cup\lambda_3=\lambda$ and $y,z$ are some partitions with $|y|=|x_I|$, $|z|=|x_{I^c}|$. Finally, we observe that 
\begin{equation}
\label{equ:decomp}
\phi_{\Gamma_A}=\frac{\ell(\lambda_1)!\ell(\lambda_2)!)\ell(\lambda_3)!}{\ell(\lambda)!}\frac{\prod y_i}{\ell(y)!}\frac{\prod z_i}{\ell(z)!}\phi_{\Gamma_A^1}\phi_{\Gamma_A^2}\phi_{\Gamma_A^3}.
\end{equation}
and
\begin{equation}
\binom{\ell(\lambda)}{s,t,u}=\binom{\ell(\lambda)}{\ell(\lambda_1),\ell(\lambda_2),\ell(\lambda_3)}=\frac{\ell(\lambda)!}{\ell(\lambda_1)!\ell(\lambda_2)!\ell(\lambda_3)!}
\end{equation} 
which cancels with the factor in \cref{equ:decomp}. This completes the proof of \cref{thm:wall}.

\section{A refined recursion for (strictly) monotone double Hurwitz numbers}
\label{sec:rec}
In this section, we derive recursive formulae for $\vec{h}_{g;x,\lambda}^{\le}$ and $\vec{h}_{g;x,\lambda}^{<}$. We then generalise these results for mixed usual/monotone/strictly monotone Hurwitz numbers.

%\alert{Fix $x^+,x^-$} \cref{eq:gwmunu}.
\begin{theorem}
\label{thm:rec}
Let $\mu$ and $\nu$ be partitions of some positive integer $d$. Moroever, let $g$ be a non-negative integer. Furthermore, we fix an ordered partition $\lambda=(\lambda_1,\dots,\lambda_k)$ with $|\lambda|=2g-2+\ell(\mu)+\ell(\nu)$ and denote $\lambda'=(\lambda_1,\dots,\lambda_{k-1})$. Then we have
\begin{align}
\vec{h}_{g;(\mu,-\nu),\lambda}^{\le,\circ}=&
\,\, \frac{1}{k} \!\!\! \sum_{\substack{I,n,\mu^i,\\\nu^i,\lambda^i,\gamma^i,g_i\\\nu'}}\prod_{i=1}^n\vec{h}_{g_i;(\mu^i,(-\nu^i,-\gamma^i)),\lambda^i}^{\le, \circ}
\frac{1}{|\mathrm{Aut}(\nu_I)|}\cdot\left(\prod_{j=1}^n\prod_{l=1}^{\ell(\gamma^j)}(\gamma^j)_l\right) \times
\\
& \times \sum_{\substack{g_1^k+g_2^k=\\\frac{\lambda_k+2-|I|+\ell(\gamma)}{2}}}\cor{ \!\! \tau_{2g^k_2 - 2} \!\! }_{g_2^k}^{\mathbb{P}^1, \circ} \cor{ \!\!(\gamma^1,\dots,\gamma^n),   \tau_{2g^k_1 - 2 + \sum\ell(\gamma^i)+\ell(\nu')} , \nu'\!\!  }^{\mathbb{P}^1,\circ}_{g_1^k},
\end{align}
and
\begin{align}
\vec{h}_{g;(\mu,-\nu),\lambda}^{<,\circ}=&
\,\, \frac{1}{k} \!\!\!
\sum_{\substack{I,n,\mu^i,\\\nu^i,\lambda^i,\gamma^i,g_i\\\nu'}}\prod_{i=1}^n\vec{h}_{g_i;(\mu^i,(-\nu^i,-\gamma^i)),\lambda^i}^{<, \circ}\cdot 
\frac{1}{|\mathrm{Aut}(\nu_I)|}\cdot \left(\prod_{i=1}^n\prod_{j=1}^{\ell(\gamma^i)}(\gamma^i)_j\right) (-1)^{\sum\ell(\gamma^j)+\ell(\nu')} \times \\
&\times \sum_{\substack{g_1^k+g_2^k=\\\frac{\lambda_k+2-|I|+\ell(\gamma)}{2}}}\cor{ \!\! \tau_{2g^k_2 - 2} \!\! }_{g_2^k}^{\mathbb{P}^1, \circ} \cor{ \!\!(\gamma^1,\dots,\gamma^n),   \tau_{2g^k_1 - 2 + \sum\ell(\gamma^i)+\ell(\nu')} , \nu'\!\!  }^{\mathbb{P}^1,\circ}_{g_1^k} ,
\end{align}
where in both formulas, the first sum is over all
\begin{enumerate}
\item subsets $I\subset\{1,\dots,\ell(\nu)\}$,
\item positive integers $n$,
\item decompositions of $\mu$, $\nu$ and $\lambda$ into $n$ partitions $\mu^1\cup\dots\cup\mu^n=\mu$, $\nu^1\cup\dots\cup\nu^n\cup\nu'=\nu$ and $\lambda^1\cup\dots\cup\lambda^n=\lambda'$, where the $\mu^i$ must be non-empty,
\item partitions $\gamma^i$ of $|\mu^i|-|\nu^i|$, where $\gamma^i$ must be non-empty,
\item non-negative integers $g_i$ with $\sum g_i=g-1+\frac{\lambda_k+2-n}{2}+\frac{3}{2}\sum\gamma^i$.
\end{enumerate}
up to order.
\end{theorem}

\begin{proof}
This result is a consequence of \cref{thm:trop}. We focus on the case of monotone Hurwitz numbers, as the argument for strictly monotone Hurwitz numbers is the same up to a sign. The idea is to consider all covers contributing to $\vec{h}_{g;(\mu,-\nu),\lambda}^{\le}$ and removing the last inner vertex which we denote by $w$. Let $\pi:\Gamma\to\mathbb{P}^1_{trop}$ be such a cover.  When we remove the last inner vertex (and thus the adjacent ends which are indexed by $I$), the cover decomposes in possibly many disconnected components. Let $n$ be their number. Each such component yields again a tropical cover $\pi^i:\Gamma^i\to\mathbb{P}^1_{trop}$ mapping to some subset $S^i\subset\{p_1,\dots,p_b\}$. Each cover $\pi^i$ is contained in $\Gamma(\mathbb{P}^1_{trop},g_i;(\mu^i,-\delta^i),\lambda^i)$ some non-negative integer $g_i$, a subpartition $\mu^i$ of $\mu$, a partition $\delta^i$ of $|\mu^i|$ and a subpartition $\lambda^i$ of $\lambda$. We note that $\delta^i$ can be decomposed into a subpartition of $\nu$ which we denote by $\nu^i$ and some partition $\gamma^i$ given by the weights of the edges adjacent to the removed vertex and contained in the $i-$th component, i.e. we have $\delta^i=(\nu^i,\gamma^i)$. This data satisfies conditions (1)--(5) stated in the theorem. The first four conditions are immediate. In order to observe the fifth condition, we consider the Euler characteristics of the graphs $\Gamma$ and $\Gamma_i$. The Euler characteristic of $\Gamma$ is given by
\begin{equation}
\label{equ:rec0}
|V(\Gamma)|-|E(\Gamma)|=1-b_1(\Gamma)=1-g+\sum_{v\in V^{in}(\Gamma)}g(v).
\end{equation}
and the Euler charcteristic of $\Gamma_i$ is given by
\begin{equation}
\label{equ:rec1}
|V(\Gamma_i)|-|E(\Gamma_i)|=1-b_1(\Gamma_i)=1-g+\sum_{v\in V^{in}(\Gamma_i)}g(v).
\end{equation}
However, we see that
\begin{equation}
\label{equ:rec2}
\left(|V(\Gamma)|-1-|I|+\sum \ell(\gamma^i)\right)-(|E(\Gamma)-|I|)=\sum_i|V(\Gamma_i)|-\sum |E(\Gamma_i)|,
\end{equation}
since we remove a single vertex and $|I|$ many ends and leaves attached to it, i.e. $|I|$ vertices and $|I|$ edges. Moreover, all incoming edges of the removed vertices obtain an additional vertex which yields $\sum_i\ell(\gamma^i)$ many vertices. By combining \cref{equ:rec0}, \cref{equ:rec1} and \cref{equ:rec2}, we obtain
\begin{equation}
1-g+\sum_{v\in V^{in}(\Gamma)}g(v)+\sum\ell(\gamma^i)=\sum_{i=1}^n \left(1-g_i+\sum_{v\in V^{in}(\Gamma_i)}g(v)\right).
\end{equation} 
We observe that $\sum_i\sum_{v\in V^{in}(\Gamma_i)}g(v)=\sum_{v\in V(\Gamma)}g(v)-g(w)$ and therefore obtain
\begin{equation}
1-g+g(w)+\sum\ell(\gamma^i)=n-\sum g_i.
\end{equation}
However, we know that $\mathrm{val}(w)=n+\sum \ell(\gamma^i)$ and thus $g(w)=\frac{\lambda_k+2-n+\sum\ell(\gamma^i)}{2}$. Thus, we obtain
\begin{equation}
\sum g_i=g-1+\frac{\lambda_k+2-n}{2}+\frac{3}{2}\sum\ell(\gamma^i),
\end{equation}
which is the last condition.\par 
On the other hand, starting with data $I,n,\mu^i,\nu^i,\lambda^i,\gamma^i,g_i,\nu'$ satisfying these conditions, one can consider $n$ tropical covers $\pi_i:\Gamma_i\to\mathbb{P}^1_{trop}$, where $\pi_i\in\Gamma(\mathbb{P}^1_{trop},g_i,(\mu^i,(-\nu^i,\gamma^i),\lambda^i)$. We can then glue the $\pi^i$s to a cover $\pi:\Gamma\to\mathbb{P}^1_{trop}$ contributing to $\Gamma(\mathbb{P}^1_{trop},g,(\mu,-\nu),\lambda)$: first, we choose subsets $S^i$ of $\{p_1,\dots,p_{k-1}\}$ with $|S^i|=\ell(\lambda^i)$. There are $\binom{\ell(\lambda)-1}{\ell(\lambda^1),\dots,\ell(\lambda^n)}$ such choices. Then the vertices of $\pi^i$ map to the points in $S^i$, while maintaining the order of the images of the vertices in $\pi^i$. We then join the edges with weights corresponding to the partitions $\gamma^i$ to a single vertex $w$, such that these edges are incoming edges and $w$ maps to $p_k$. Moreover, we attach $\ell(\nu')$ outgoing edges to $w$ which are ends with weights in bijection to the entries of $\ell(\nu')$. This way, we obtain a cover $\pi\in\Gamma(\mathbb{P}^1_{trop},g,(\mu,-\nu),\lambda)$. Let $\omega(\Gamma),\omega(\Gamma_i)$ be the weight of the graphs $\Gamma$ and $\Gamma_i$. Then we observe that
\begin{align}
\omega(\Gamma)=&\frac{\prod\ell(\lambda^i)!}{\ell(\lambda)!}\cdot\frac{1}{|\mathrm{Aut}(\nu_I)|}\cdot\prod\omega(\Gamma_i)\cdot\left(\prod_{i=1}^n\prod_{j=1}^{\ell(\gamma^i)}(\gamma^i)_j\right)\\
&\sum_{\substack{g_1^k+g_2^k=\\\frac{\lambda_k+2-|I|+\ell(\gamma)}{2}}}\cor{ \!\! \tau_{2g^k_2 - 2} \!\! }_{g_2^k}^{\mathbb{P}^1, \circ} \cor{ \!\!(\gamma^i,\dots,\gamma^n),   \tau_{2g^k_1 - 2 + + \sum\ell(\gamma^i)+\ell(\nu')} , \nu'\!\!  }^{\mathbb{P}^1,\circ}_{g_1^k}
\end{align}
where we note that $\frac{1}{|\mathrm{Aut}(\nu_I)|}$ contributes to $\frac{1}{|\mathrm{Aut}(\Gamma)|}$. This completes the proof.
\end{proof}

%%%%%%%%%%%%%%%%%%%%
%Mixed Hurwitz
We want now to generalise the statement above to mixed Hurwitz numbers. The following definition expresses mixed $p$-strictly monotone/ $q$-monotone/ $(b - (p+q))$-usual double Hurwitz numbers in terms of tropical covers weighted by Gromov-Witten invariants.

\begin{definition}
Let $g$ be a non-negative integer, and  $x\in\left(\mathbb{Z}\backslash\{0\}\right)^n$ wih $|x^+|=|x^-|=d$, $b = 2g - 2 + n$, let $p$ and $q$ be two integers such that $p + q \leq b$.
Let $\lambda_{(1)}$ be a partition of $p$ and let $\lambda_{(2)}$ be a partition of $q$, set $\tilde{\lambda}_i := 1$ for $i = 1, \dots, b - (p+q)$ and finally set $\lambda := \lambda_{(1)} \cup \lambda_{(2)} \cup \tilde{\lambda} $. We are ready to define the $\lambda$-slice of the mixed $p$-strictly monotone/ $q$-monotone/ $(b - (p+q))$-usual double Hurwitz numbers

\begin{align}
h_{g; x, p, q, \lambda}^{ \times, <, \leq, \bullet} 
&=
\sum_{\pi \in \Gamma( \mathbb{P}^1_{\text{trop}}, g; x,\lambda)}\frac{1}{|\mathrm{Aut}(\pi)|}\frac{1}{\ell(\lambda)!}\prod_{i=1}^p (-1)^{1 + \mathrm{val}(v_i)} m_{v_i} 
\prod_{j=p+1}^{p+q}  m_{v_j} 
\prod_{k=p+q + 1}^{b}  m_{v_k} 
\prod_{e \in E(\Gamma)} \omega_e,
\end{align}
where $\Gamma( \mathbb{P}^1_{\text{trop}}, g; x,\lambda)$ is the set of tropical covers 
$
\pi: \Gamma \longrightarrow \mathbb{P}^1_{trop} = \mathbb{R}
$
with $b=2g-2+n$ points $p_1,\dots,p_b$ fixed on the codomain $\mathbb{P}^1_{trop}$ and $\lambda$ an ordered partition of $b$, such that
\begin{itemize}
\item[\textit{i).}] The unbounded left (resp. right) pointing ends of $\Gamma$ have weights given by the partition $x^+$ (resp. $x^-$).
\item[\textit{ii).}] The graph $\Gamma$ has $l:=\ell(\lambda)\le b$ vertices. Let $V(\Gamma) = \{v_1, \dots, v_l\}$ be the set of its vertices. Then we have $\pi(v_i)=p_i$ for $i=1,\dots,l$. Moreover, let $w_i  = \mathrm{val}(v_i)$ be the corresponding valencies.
% such that 
%$$
%\sum_{i=1}^l w_i  =  \ell(\mu) + \ell(\nu) + l - 2
%$$
\item[\textit{iii).}] We assign an integer $g(v_i)$ as the genus to $v_i$ and the following condition holds true

\begin{equation}
h^1(\Gamma) + \sum_{i=1}^l  g(v_i) = g.
\end{equation}
\item[\textit{iv).}] We have $\lambda_i=\mathrm{val}(v_i)+2g(v_i)-2$.
\item[\textit{v).}] For each vertex $v_i$, let $y^{+}$ (resp. $ y^-$) be the tuple of weights of those edges adjacent to $v_i$ which map to the right-hand (resp. left-hand) of $p_i$. The multiplicity $m_{v_i}$ of $v_i$ is defined to be
\begin{align}
m_{v_i} = &(\lambda_i-1)!|\mathrm{Aut}(y^{+})||\mathrm{Aut}(y^{-})|\\
&\sum_{g_1^i+g_2^i=g(v_i)}\cor{ \!\! \tau_{2g^i_2 - 2}(\omega) \!\! }_{g^i_2}^{\mathbb{P}^1, \circ} \cor{ \!\!y^+,   \tau_{2g^i_1 - 2 + n}(\omega) , y^-\!\!  }^{\mathbb{P}^1,\circ}_{g^i_1}
\end{align}
\end{itemize}
Note that the $m_{v_k}$ above always simplify to either one (in most of the cases) or two (only in case the two half-edges directed towards the same end have equal weights). Furthermore, we define $h_{g;x, p,q, \lambda}^{\times, <, \le,\circ}$  by considering only connected source curves.
\end{definition}

%%%%%%%

\begin{remark}
It is a straightforward generalisation of theorem \label{thm:trop} in \cite{HL} the fact that these numbers 
$h_{g;x, p,q, \lambda}^{\times,<,\leq, \bullet}$
are
the $\lambda$-slices of mixed usual/monotone/strictly-monotone Hurwitz numbers, meaning that if we define

\begin{equation}
h_{g;x, p,q}^{\times,<,\leq, \bullet}
\coloneqq
 \sum_{\substack{\lambda  = (\lambda_{(1)}, \lambda_{(2)}, \tilde{\lambda}) \vdash b 
 \\
 \lambda_i = 1, \,\, i = p+q+1, \dots, b
 \\
 \lambda_{(1)} \vdash p, \,\,\,
 \lambda_{(2)} \vdash q
} }
h_{g;x, p,q, \lambda}^{\times,<,\leq, \bullet}
\end{equation}
then $h_{g;x, p,q}^{\times,<,\leq, \bullet}$ enumerates all weighted ramified covers of degree $d = |x^+| = |x^-| $ of the Riemann sphere by genus $g$ compact surfaces where the ramification profiles over zero and infinity are given by $x^{+}$ and $x^{-} $, respectively, and all other ramifications are simple (and therefore can be represented as transpositions $(a_i, b_i)_{i=1, \dots, b}$ with $1 \leq a_i < b_i \leq d$), in such a way that the first $p$ simple ramifications satisfy the strictly monotone condition, the following $q$ satisfy the weakly monotone condition, and the remaining $b - (p + q)$ are usual simple ramifications (and hence do not satisfy any additional requirement):
\begin{enumerate}
\item $b_i < b_{i+1}, \qquad $ for $i = 1, \dots, \, p-1$,
\item $b_i \leq b_{i+1}, \qquad $ for $i = p+1, \dots, \, p+q-1$.
\end{enumerate}
\end{remark}
With the notations above, we are going to generalise theorem \ref{thm:rec} by cutting one vertex of the tropical covers. However, there are now three different types of vertices, as opposed to one in \ref{thm:rec}: the strictly monotone vertices, the weakly monotone vertices, and the usual vertices. We therefore obtain three different recursions, depending on which type of vertex we are cutting. Note that the first and the second type of vertex differ just by a sign factor in their weights, whereas the third type is extremely simple as its genus is zero and its cardinality must be equal to three. It is moreover possible to have first and second type of vertices which happen to be usual vertices ( this happens if and only if they come from parts of $\lambda$ equal to one in the first $p+q$ parts): we still treat them according to their general nature, as the formula for their weight in that case naturally specialises to the weight of usual vertices.

\begin{corollary}
\label{cor:mixedrec}
Let $\mu$ and $\nu$ be partitions of some positive integer $d$, let $ g, p, q $ be a non-negative integers, let $\lambda$ be a partition $\lambda= \lambda_{(1)} \cup \lambda_{(2)} \cup \tilde{\lambda} = (\lambda_1,\dots,\lambda_k)$ with $|\lambda|= b = 2g-2+\ell(\mu)+\ell(\nu)$, $\lambda_{(1)} \vdash p, \, \lambda_{(2)}\vdash q, \, p+q \leq b, \, \tilde{\lambda}_i = 1$ for all $i$, and for a partition $\sigma$ denote $\sigma' = \sigma \setminus \{ \sigma_{\ell(\sigma)} \}$.
 Then we have the following three recursions:
 \begin{enumerate}
 \item[\textit{i).}] 
 \textbf{Cutting along a strictly monotone vertex.}
\begin{align}
\vec{h}_{g;(\mu,-\nu), p,q, \lambda}^{\times, <, \leq, \circ}=
 & \,\, \frac{1}{k} \!\!\!
 \sum_{
 \substack{
 I,n,\mu^i,\nu^i,\nu' 
 \\ 
 \lambda_{(1)}^i, \lambda_{(2)}^i, \tilde{\lambda}^i,\gamma^i,g_i
 }
 }
 \prod_{i=1}^n\vec{h}_{g_i;(\mu^i,(-\nu^i,-\gamma^i)),\lambda^i}^{<, \circ}\cdot (-1)^{\sum\ell(\gamma^j)+\ell(\nu')}
\cdot\frac{1}{|\mathrm{Aut}(\nu_I)|}\cdot \left(\prod_{i=1}^n\prod_{j=1}^{\ell(\gamma^i)}(\gamma^i)_j\right)\\
&\sum_{\substack{g_1^k+g_2^k=\\\frac{\lambda_k+2-|I|+\ell(\gamma)}{2}}}\cor{ \!\! \tau_{2g^k_2 - 2} \!\! }_{g_2^k}^{\mathbb{P}^1, \circ} \cor{ \!\!(\gamma^i,\dots,\gamma^n),   \tau_{2g^k_1 - 2 + + \sum\ell(\gamma^i)+\ell(\nu')} , \nu'\!\!  }^{\mathbb{P}^1,\circ}_{g_1^k},
\end{align}
 \item[\text{ii).}]
 \textbf{Cutting along a weakly monotone vertex.}
\begin{align}
\vec{h}_{g;(\mu,-\nu), p,q, \lambda}^{\times, <, \leq, \circ}=
 & \,\, \frac{1}{k} \!\!\!
 \sum_{
 \substack{
 I,n,\mu^i,\nu^i,\nu' 
 \\ 
 \lambda_{(1)}^i, \lambda_{(2)}^i, \tilde{\lambda}^i,\gamma^i,g_i
 }
 }
 \prod_{i=1}^n\vec{h}_{g_i;(\mu^i,(-\nu^i,-\gamma^i)),\lambda^i}^{<, \circ}
\cdot\frac{1}{|\mathrm{Aut}(\nu_I)|}\cdot \left(\prod_{i=1}^n\prod_{j=1}^{\ell(\gamma^i)}(\gamma^i)_j\right)\\
&\sum_{\substack{g_1^k+g_2^k=\\\frac{\lambda_k+2-|I|+\ell(\gamma)}{2}}}\cor{ \!\! \tau_{2g^k_2 - 2} \!\! }_{g_2^k}^{\mathbb{P}^1, \circ} \cor{ \!\!(\gamma^i,\dots,\gamma^n),   \tau_{2g^k_1 - 2 + + \sum\ell(\gamma^i)+\ell(\nu')} , \nu'\!\!  }^{\mathbb{P}^1,\circ}_{g_1^k},
\end{align}
 \item[\text{iii).}]
 \textbf{Cutting along a usual vertex.}
\begin{align}
\vec{h}_{g;(\mu,-\nu), p,q, \lambda}^{\times, <, \leq, \circ}=
 & \,\, \frac{1}{k} \!\!\!
 \sum_{
 \substack{
 I,n \leq 2,\mu^i,\nu^i,\nu' 
 \\ 
 \lambda_{(1)}^i, \lambda_{(2)}^i, \tilde{\lambda}^i,\gamma^i,g_i
 }
 }
 \prod_{i=1}^n\vec{h}_{g_i;(\mu^i,(-\nu^i,-\gamma^i)),\lambda^i}^{<, \circ}
\cdot\frac{1}{|\mathrm{Aut}(\nu_I)|}\cdot \left(\prod_{i=1}^n\prod_{j=1}^{\ell(\gamma^i)}(\gamma^i)_j\right)
\end{align}
\end{enumerate}
where in all three formulas, the first sum is over all
\begin{enumerate}
\item subsets $I\subset\{1,\dots,\ell(\nu)\}$,
\item positive integers $n$ (smaller or equal than $2$ in the third recursion),
\item decompositions of $\mu$, $\nu$ and $\lambda$ into $n$ partitions 
$
\mu^1\cup\dots\cup\mu^n=\mu, \nu^1\cup\dots\cup\nu^n\cup\nu'=\nu,
$
where the $\mu^i$ must be non-empty,
\item partitions $\gamma^i$ of $|\mu^i|-|\nu^i|$, where $\gamma^i$ must be non-empty, 
\item non-negative integers $g_i$ with $\sum g_i=g-1+\frac{\lambda_k+2-n}{2}+\frac{3}{2}\sum\gamma^i$,
\item in the third case we require $|\nu'| = 3 - \sum_i \ell(\gamma^i)$.
\end{enumerate}
up to order, and moreover
\begin{enumerate}
\item[\textit{i).}] when cutting over a strictly monotone vertex we have
 $$ 
 \lambda_{(1)}^1\cup\dots\cup\lambda_{(1)}^n=\lambda_{(1)}', 
 \qquad
 \lambda_{(2)}^1\cup\dots\cup\lambda_{(2)}^n=\lambda_{(2)},
 \qquad
 \tilde{\lambda}^1 \cup \dots \cup \tilde{\lambda}^n=\tilde{\lambda}.
 $$
 \item[\textit{ii).}] when cutting over a weakly monotone vertex we have
 $$ 
 \lambda_{(1)}^1\cup\dots\cup\lambda_{(1)}^n=\lambda_{(1)}, 
 \qquad
 \lambda_{(2)}^1\cup\dots\cup\lambda_{(2)}^n=\lambda_{(2)}',
 \qquad
 \tilde{\lambda}^1 \cup \dots \cup \tilde{\lambda}^n=\tilde{\lambda}.
 $$
 \item[\textit{iii).}] when cutting over a usual vertex we have
 $$ 
 \lambda_{(1)}^1\cup\dots\cup\lambda_{(1)}^n=\lambda_{(1)}, 
 \qquad
 \lambda_{(2)}^1\cup\dots\cup\lambda_{(2)}^n=\lambda_{(2)},
 \qquad
 \tilde{\lambda}^1 \cup \dots \cup \tilde{\lambda}^n=\tilde{\lambda}'.
 $$
\end{enumerate}
\begin{proof}
The proof is a straightforward generalisation of the one of theorem \ref{thm:rec}. The main difference is that we need to keep track of the partitions of $p$ and $q$ when cutting, and eliminate the right cut vertex from the summations over $\lambda_{(1)}^i, \lambda_{(2)}^i, \tilde{\lambda}^i$. The cut vertex in third recursion has genus zero and valency exactly three: the recursion has trivial residue Gromov-Witten invariants and $n$ gets bounded by two. The extra signs appear only in the first recursion, when we cut a vertex of strict monotone type. This concludes the proof of the corollary.
\end{proof}
\end{corollary}

\bibliographystyle{acm}
\bibliography{literature.bib}

\end{document}